\documentclass[12pt]{amsart}
\usepackage{mathtools} % basic maths packages
\usepackage{dsfont} % for mathematical fonts like mathbb{R}
\usepackage{tensor} % for tensor indices

\usepackage[pdftex]{graphicx} % for including pictures with includegraphics
\usepackage{color}
\usepackage[colorlinks,citecolor=blue,linkcolor=blue,urlcolor=blue]{hyperref} % for references in the text with colors
\usepackage{cleveref}

\usepackage[utf8]{inputenc} % to translate accents to intern latex code
\usepackage[english]{babel} % for english language

\usepackage{geometry} % for size, margins, ...
\usepackage{tikz} % for pictures with tikz
\usepackage{cite} %bibtex
\usepackage{enumitem} %indent of the tables in Appendix
\usepackage{musicography} % for musical signs
\usepackage{pgfplots}
\usepgfplotslibrary{polar}
\pgfplotsset{compat=1.15}
\usepackage{mathrsfs}
\usetikzlibrary{arrows}

\newtheorem{theorem}{Theorem}

\newtheorem{proposition}[theorem]{Proposition}
\newtheorem{example}[theorem]{Example}
\newtheorem{Remark}[theorem]{Remark}
\newtheorem{question}[theorem]{Question}

\newtheorem{observation}[theorem]{Observation}

\numberwithin{theorem}{section} % theorem labeling within sections
\numberwithin{equation}{section} % equation labeling within sections
\numberwithin{figure}{section} % figure labeling within sections

\newcounter{question}

\newcommand{\q}{\stepcounter{question}\medskip \noindent\textbf{\thequestion.}\,}

\newcommand*{\R}{\mathbb{R}}

\newcommand{\degre}{\ensuremath{^\circ}}

\title{Circular isoptics in Flatland}
\author{Alexander Thomas}
\address{Universit\'e Lyon 1, 21 av. Claude Bernard, 69100 Villeurbanne}
\email{athomas@math.univ-lyon1.fr}
\begin{document}

\begin{abstract}
We explore convex shapes $S$ in the Euclidean plane which have the following property: there is a circle $C$ such that the angle between the two tangents from any point of $C$ to $S$ is constant equal to $\alpha$. A dynamical formulation allows to analyze the existence of such shapes. Interestingly, the existence of non-circular shapes depends in a non-trivial way on the angle $\alpha$.
\end{abstract}

\maketitle

\section{Flatland}

Imagine how a creature living in Edwin Abbott's \emph{Flatland} sees their world. Suppose it has only one eye, then any connected shape is just a line segment in his view, which can be described by an angle, the \emph{angle of sight}.
A natural question is:

\medskip
\textit{Can a Flatlander determine the shape of an object, by going around it on a circle, just via the angles of sight?}

\begin{figure}[h!]
\centering
\definecolor{uuuuuu}{rgb}{0.26666666666666666,0.26666666666666666,0.26666666666666666}
\definecolor{qqwuqq}{rgb}{0,0.39215686274509803,0}
\definecolor{xdxdff}{rgb}{0.49019607843137253,0.49019607843137253,1}

\begin{tikzpicture}[line cap=round,line join=round,>=triangle 45,x=1cm,y=1cm, scale=0.5]
%\clip(-5.686470995670994,-5.759473593073588) rectangle (9.9152606060606,6.153946320346323);
\draw [shift={(3.594834558253917,3.821002671226993)},line width=1pt,color=qqwuqq,fill=qqwuqq,fill opacity=0.10000000149011612] (0,0) -- (-152.3135167987636:0.8) arc (-152.3135167987636:-97.50632855234342:0.8) -- cycle;
\draw [shift={(5.312616862246874,-1.8556192785936363)},line width=1pt,color=qqwuqq,fill=qqwuqq,fill opacity=0.10000000149011612] (0,0) -- (128.90740035367557:0.8) arc (128.90740035367557:164.68496738114482:0.8) -- cycle;
\draw [shift={(-2.9925169478307003,-1.3579827418512065)},line width=1pt,color=qqwuqq,fill=qqwuqq,fill opacity=0.10000000149011612] (0,0) -- (7.209730191692466:0.8) arc (7.209730191692466:48.922955117216794:0.8) -- cycle;
\draw [line width=1pt,dashed] (1.2572519480519464,0.015418181818185744) circle (4.466180220898789cm);
\draw [rotate around={161.6475452560356:(1.2541550291590398,0.6884256482047278)},line width=1.5pt] (1.2541550291590398,0.6884256482047278) ellipse (1.9281188483158058cm and 1.3799117086939598cm);
\draw [line width=1pt] (-2.9925169478307003,-1.3579827418512065)-- (1.9768798982059703,-0.7293445181236582);
\draw [line width=1pt] (5.312616862246874,-1.8556192785936363)-- (0.9202168828063693,-0.6527558544400027);
\draw [line width=1pt] (5.312616862246874,-1.8556192785936363)-- (2.689346003523837,1.3945768517205106);
\draw [line width=1pt] (3.594834558253917,3.821002671226993)-- (3.0920857683652714,0.005503148835806382);
\draw [line width=1pt] (3.594834558253917,3.821002671226993)-- (-0.009121654744113528,1.9299677949752825);
\draw [line width=1pt] (-2.9925169478307003,-1.3579827418512065)-- (-0.41057165248979777,1.604151560264207);
\draw [fill=xdxdff] (3.594834558253917,3.821002671226993) circle (2.5pt);
\draw[color=uuuuuu] (3.8459965367965347,4.37039653679654) node {$P_3$};
\draw [fill=xdxdff] (-2.9925169478307003,-1.3579827418512065) circle (2.5pt);
\draw[color=uuuuuu] (-3.7,-1.2573090909090867) node {$P_1$};
\draw [fill=xdxdff] (5.312616862246874,-1.8556192785936363) circle (2.5pt);
\draw[color=uuuuuu] (6.2,-1.7) node {$P_2$};
\draw[color=qqwuqq] (2.7,2.6) node {$55^\circ$};
\draw[color=qqwuqq] (3.5,-0.8) node {$36^\circ$};
\draw[color=qqwuqq] (-1.4,-0.6685645021644981) node {$42^\circ$};
\draw[color=uuuuuu] (1.5,0.8) node {$S?$};
\end{tikzpicture}
%\vspace*{-0.7cm}
\caption{Angles of sight}\label{Fig:angle-sight}
\end{figure}
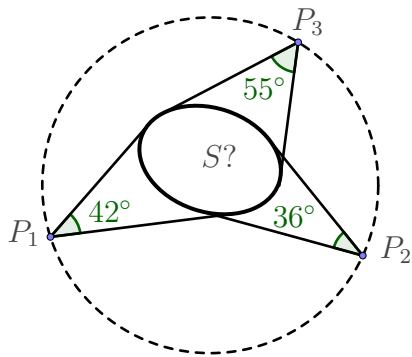

To be more precise, let $S\subset \R^2$ be closed and convex. For a point $P\notin S$, the union of all rays $PQ$ where $Q$ varies in $S$ is a cone centered at $P$. The \emph{angle of sight} at $P$ is defined to be the angle of that cone. It is the angle between the two tangent lines from $P$ to $S$.

Fix a circle $C$ with center $O$ on which the Flatlander walks, while looking at a convex shape $S$ inside the disk bounded by $C$. The question is whether $S$ can be determined by the angles $\alpha(P)$ for $P$ varying on $C$.

An interesting special case is when all angles are constant, equal to a fixed $\alpha\in (0,\pi)$. One possible shape $S$ having constant angle of sight is of course a disk centered at $O$. But is it the only possible shape?

\begin{question}\label{Q1}
Are there ``constant angle shapes'' other than disks?
\end{question}

The question is similar to the question of convex shapes with constant width. Indeed, when the radius of $C$ goes to infinity, the tangent lines from $P\in C$ to $S$ become parallel. It is natural to replace the angle of sight $\alpha$ by the distance of the parallel lines. So the question becomes: are there convex shapes having constant diameter in all directions?

The answer to this is well-known: there is a great variety of shapes with constant shapes, the most famous being the \emph{Reuleaux triangle}, which has minimal area among all shapes with constant and fixed width \cite{blaschke1915konvexe}.

\begin{figure}[h!]
\centering
\definecolor{ffzzqq}{rgb}{1,0.6,0}
\definecolor{ffzztt}{rgb}{1,0.6,0.2}

\begin{tikzpicture}[line cap=round,line join=round,>=triangle 45,x=1cm,y=1cm, scale=0.7]
%\clip(-5.625864935064942,-2.8590406926406873) rectangle (4.936905627705617,3.9461541125541153);
\fill[line width=2pt,color=ffzztt,fill=ffzztt,fill opacity=1] (-3.028462337662346,-0.746486580086576) -- (1.5083341991341892,-0.746486580086576) -- (-0.7600640692640772,3.1824944725804865) -- cycle;
\draw [line width=2pt,dash pattern=on 1pt off 1pt,color=ffzztt] (1.5083341991341892,-0.746486580086576)-- (-0.7600640692640772,3.1824944725804865);
\draw [line width=2pt,dash pattern=on 1pt off 1pt,color=ffzztt] (-0.7600640692640772,3.1824944725804865)-- (-3.028462337662346,-0.746486580086576);
\draw [shift={(-3.028462337662346,-0.746486580086576)},line width=2pt,color=ffzztt,fill=ffzztt,fill opacity=1]  plot[domain=0:1.0471975511965976,variable=\t]({1*4.536796536796535*cos(\t r)+0*4.536796536796535*sin(\t r)},{0*4.536796536796535*cos(\t r)+1*4.536796536796535*sin(\t r)});
\draw [shift={(-0.7600640692640772,3.1824944725804865)},line width=2pt,color=ffzzqq,fill=ffzzqq,fill opacity=1]  plot[domain=4.1887902047863905:5.235987755982988,variable=\t]({1*4.536796536796536*cos(\t r)+0*4.536796536796536*sin(\t r)},{0*4.536796536796536*cos(\t r)+1*4.536796536796536*sin(\t r)});
\draw [shift={(1.5083341991341892,-0.746486580086576)},line width=2pt,color=ffzzqq,fill=ffzzqq,fill opacity=1]  plot[domain=2.0943951023931957:3.141592653589793,variable=\t]({1*4.536796536796535*cos(\t r)+0*4.536796536796535*sin(\t r)},{0*4.536796536796535*cos(\t r)+1*4.536796536796535*sin(\t r)});
\draw [line width=1pt,dashed] (-3.028462337662346,-0.746486580086576)-- (1.5083341991341892,-0.746486580086576);
\draw [line width=1pt,dashed] (1.5083341991341892,-0.746486580086576)-- (-0.7600640692640772,3.1824944725804865);
\draw [line width=1pt,dashed] (-0.7600640692640772,3.1824944725804865)-- (-3.028462337662346,-0.746486580086576);
\end{tikzpicture}
%\vspace*{-1.5cm}
\caption{Reuleaux triangle}\label{Fig:Reuleaux}
\end{figure}
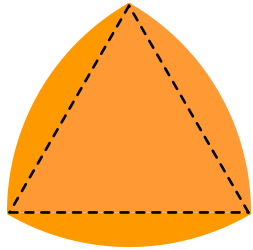

So it might not come as a surprise that there are non-trivial examples of constant angle shapes. Take for instance an ellipse $E$ and consider the set of points such that the two tangent lines to $E$ are perpendicular. This set is a circle, well-known as the \emph{orthoptic circle} \cite[Prop. 7.21]{sidler}. The same example works for any conic other than the parabola (where the orthoptic circle degenerates into a straight line).

\begin{figure}[h!]
\centering
\definecolor{uuuuuu}{rgb}{0.26666666666666666,0.26666666666666666,0.26666666666666666}
\definecolor{qqwuqq}{rgb}{0,0.39215686274509803,0}
\definecolor{xdxdff}{rgb}{0.49019607843137253,0.49019607843137253,1}

\begin{tikzpicture}[line cap=round,line join=round,>=triangle 45,x=1cm,y=1cm, scale=0.55]
%\clip(-5.686470995670994,-5.759473593073588) rectangle (9.915260606060599,6.153946320346323);
\draw[line width=1.5pt,color=qqwuqq,fill=qqwuqq,fill opacity=0.10000000149011612] (0.47138288910134535,4.163277640349452) -- (0.755352308376035,3.9302725460346046) -- (0.9883574026908816,4.214241965309294) -- (0.7043879834161919,4.447247059624141) -- cycle; 
\draw [rotate around={-2.915322116642519:(1.2572519480519466,0.015418181818185746)},line width=1.5pt] (1.2572519480519466,0.015418181818185746) ellipse (3.424413949238913cm and 2.867081245762928cm);
\draw [line width=1.5pt,dashed] (1.2572519480519464,0.015418181818185744) circle (4.466180220898789cm);
\draw [line width=1.5pt] (-0.2778882282645647,5.2532330067394835)-- (6.63080205588211,-0.41554652368067213);
\draw [line width=1.5pt] (1.597424527534914,5.535614116877406)-- (-4.286905742078956,-1.6357735061967986);
\draw [fill=xdxdff] (0.7043879834161919,4.44724705962414) circle (2.5pt);
\draw[color=uuuuuu] (0.6858233766233759,4.959141125541128) node {$P$};
\draw[color=uuuuuu] (0,-2.1) node {$E$};
\draw[color=qqwuqq] (0.85,3.4) node {$90^\circ$};
\end{tikzpicture}
%\vspace*{-1cm}
\caption{Orthoptic circle}\label{Fig:orthoptic}
\end{figure}
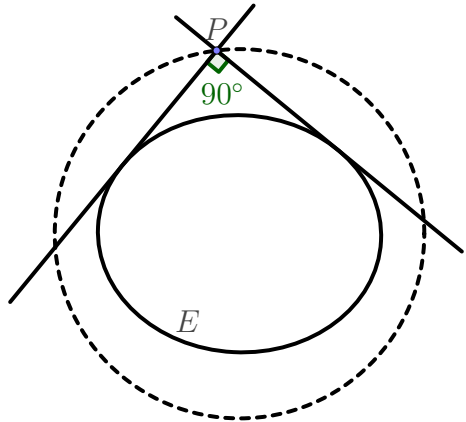

The true surprise for me was that the classification of constant angle shapes depends in an unexpected way on the angle $\alpha$. In the example of the conics above we have $\alpha=\tfrac{\pi}{2}$. If $\alpha$ is an irrational multiple of $\pi$, then there is only the disk. But for some other rational values, for example for $\alpha=\tfrac{\pi}{3}$, there still is only the disk!

A reformulation of the problem is the following: for fixed convex shape $S$ and angle $\alpha\in (0,\pi)$, consider the set of points $P$ in the plane such that the angle of sight at $P$ to $S$ is $\alpha$. This forms a simply closed curve $S_\alpha$ called an \emph{isoptics}. Varying $\alpha\in (0,\pi)$ gives a foliation of $\mathbb{R}^2\backslash S$ by the curves $S_\alpha$. The question then becomes: 

\begin{question} 
For which convex compact shapes there is a circle among the isoptics? For which $\alpha$ can this appear?
\end{question}

The complete answer to the existence of constant angle shapes is given by the following:

\begin{theorem}[Green 1949]\label{main-thm}
There exist shapes of constant angle $\alpha$ other than the disk if and only if $\alpha=\tfrac{p}{q}\pi$, where $(p,q)$ are two coprime intergers with $p<q$ and $q-p$ is odd. In that case, there are uncountably many examples.
\end{theorem}

I will present an elementary proof combining geometry and dynamical systems in Section \ref{Sec:dyn-viewpt}. Later I found out that the result was obtained a long time ago by Green \cite{green}, using the support function and some simple analysis. The dynamical viewpoint leads to some interesting billiards-like systems, which we quickly discuss in Section \ref{Sec:strange-billiards}.

This problem of constant angle shapes has been asked at a relatively new mathematical tournament, the ETEAM (for \emph{European Tournament of Enthousiastic Apprentice Mathematicians}). The main difference to classical maths olympiads is that ETEAM tries to simulate research. It is a team competition, where high school students have several months to think about (partially) open problems. I'll say more about that in the last section \ref{Sec:opening}.

\smallskip
\textit{Acknowledgements.} I'm grateful to Peter Smillie for stimulating discussions and to the referee for helpful comments.

\section{Dynamical viewpoint}\label{Sec:dyn-viewpt}

Here is another reformulation of the problem using a dynamical system.
Start from a point $P_1$ on the circle $C$ and choose one of the two tangents $t_1$ of $P_1$ to $S$. Then, $t_1$ intersects $C$ in a unique point $P_2$ different from $P_1$. Let $t_2$ be the second tangent from $P_2$ to $S$ (the other one being $t_1$), see Figure \ref{Fig:basics}. The dynamical system is then $(P_1,t_1)\mapsto (P_2,t_2)$.

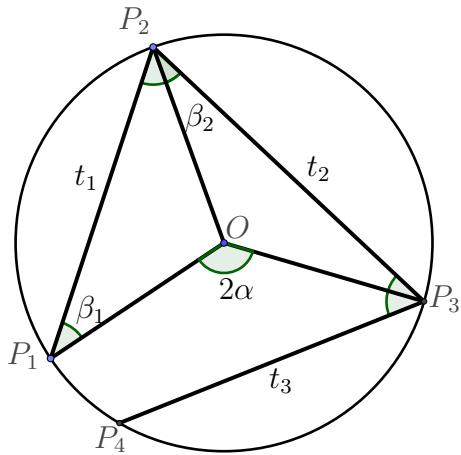
\begin{figure}[h!]
\definecolor{uuuuuu}{rgb}{0.26666666666666666,0.26666666666666666,0.26666666666666666}
\definecolor{qqwuqq}{rgb}{0,0.39215686274509803,0}
\definecolor{xdxdff}{rgb}{0.49019607843137253,0.49019607843137253,1}
\definecolor{ududff}{rgb}{0.30196078431372547,0.30196078431372547,1}

\begin{tikzpicture}[line cap=round,line join=round,>=triangle 45,x=1cm,y=1cm, scale=0.5]
%\clip(-9.57,-6.74) rectangle (6.29,7.1);
\draw [shift={(-3.7543361475714274,5.631993599665648)},line width=1pt,color=qqwuqq,fill=qqwuqq,fill opacity=0.10000000149011612] (0,0) -- (-108.21653013970847:1) arc (-108.21653013970847:-43.21653013970846:1) -- cycle;
\draw [shift={(3.449130179005303,-1.1364235088159713)},line width=1pt,color=qqwuqq,fill=qqwuqq,fill opacity=0.10000000149011612] (0,0) -- (136.78346986029158:1) arc (136.78346986029158:201.78346986029158:1) -- cycle;
\draw [shift={(-6.48136055344257,-2.654240368961713)},line width=1pt,color=qqwuqq,fill=qqwuqq,fill opacity=0.10000000149011612] (0,0) -- (33.690067525979785:1) arc (33.690067525979785:71.78346986029156:1) -- cycle;
\draw [shift={(-3.7543361475714274,5.631993599665648)},line width=1pt,color=qqwuqq,fill=qqwuqq,fill opacity=0.10000000149011612] (0,0) -- (-70.12312780539669:1) arc (-70.12312780539669:-43.21653013970845:1) -- cycle;
\draw [line width=1pt] (-1.87,0.42) circle (5.5421656416963945cm);
\draw [line width=1.5pt] (-6.48136055344257,-2.654240368961713)-- (-3.7543361475714274,5.631993599665648);
\draw [line width=1.5pt] (-3.7543361475714274,5.631993599665648)-- (-1.87,0.42);
\draw [line width=1.5pt] (-1.87,0.42)-- (-6.48136055344257,-2.654240368961713);
\draw [line width=1.5pt] (-1.87,0.42)-- (3.449130179005303,-1.1364235088159713);
\draw [line width=1.5pt] (-3.7543361475714274,5.631993599665648)-- (3.449130179005303,-1.1364235088159713);
\draw [line width=1.5pt] (3.449130179005303,-1.1364235088159713)-- (-4.651390806452168,-4.3736901424459385);
\draw [shift={(-1.87,0.42)},line width=1pt,color=qqwuqq,fill=qqwuqq,fill opacity=0.10000000149011612] (0,0) -- (-146.30993247402023:0.8) arc (-146.30993247402023:-16.309932474020215:0.8) -- cycle;

\draw (-6.1,-0.7) node[anchor=north west] {$\beta_1$};
\draw (-3.2,4.4) node[anchor=north west] {$\beta_2$};
\draw (-2.3,-0.25) node[anchor=north west] {$2\alpha$};
\draw [fill=ududff] (-1.87,0.42) circle (2.5pt);
\draw[color=uuuuuu] (-1.5,0.85) node {$O$};
\draw [fill=xdxdff] (-6.48136055344257,-2.654240368961713) circle (2.5pt);
\draw[color=uuuuuu] (-7.21,-2.57) node {$P_1$};
\draw [fill=xdxdff] (-3.7543361475714274,5.631993599665648) circle (2.5pt);
\draw[color=uuuuuu] (-4.25,6.33) node {$P_2$};
\draw[color=black] (-5.53,2.15) node {$t_1$};
\draw [fill=uuuuuu] (3.449130179005303,-1.1364235088159713) circle (2pt);
\draw[color=uuuuuu] (4,-1.07) node {$P_3$};
\draw [fill=uuuuuu] (-4.651390806452168,-4.3736901424459385) circle (2pt);
\draw[color=uuuuuu] (-4.89,-4.8) node {$P_4$};
\draw[color=black] (0.65,2.39) node {$t_2$};
\draw[color=black] (-0.35,-3.25) node {$t_3$};
\end{tikzpicture}

\caption{Dynamical viewpoint}\label{Fig:basics}
\end{figure}

This system is quite similar to a billiards, but with the strange rule that the rebound at the boundary always happens at a constant angle $\alpha$.
The trajectories of this billiards-like system are all tangent to $S$. We will recover $S$ from the set of all its tangents.

Instead of considering $(P_1,t_1)$, denote by $\beta_k$ the angle between $t_k$ and $OP_k$, where $O$ is the center of the circle $C$ and $k\in\mathbb{N}$. The dynamical system $T$ is then $T:(P_k,\beta_k)\mapsto (P_{k+1},\beta_{k+1})$.

Some very elementary geometric observations will help us significantly:

\begin{observation}\label{obs-2-periodic}
The sequence $(\beta_k)$ is 2-periodic, i.e. $\beta_{k+2}=\beta_k$ for all $k\geq 1$.
\end{observation}
\begin{proof}
Consider three consecutive points $P_{k}, P_{k+1}$ and $P_{k+2}$ (see Figure \ref{Fig:basics}). Since the rebounce has constant angle $\alpha$ and triangle $P_kP_{k+1}O$ is isosceles, we get $\beta_k+\beta_{k+1}=\alpha$. This is true for all $k$, thus $\beta_k=\beta_{k+2}$.
\end{proof}

As a consequence, we see that the distance between $O$ and $P_kP_{k+1}$ is 2-periodic, since it equals $r\sin(\beta_k)$, where $r$ is the radius of $C$.

\begin{observation}
The second iteration $T^2$ of the dynamical system is a rotation around $O$ by angle $2\alpha$.
\end{observation}

This simply follows from the first observation and the inscribed angle theorem. We are now ready to discuss the first case of the Theorem \ref{main-thm}:

\begin{proposition}
If $\alpha$ is an irrational multiple of $\pi$, then the only convexe shape with constant angle $\alpha$ is the disk.
\end{proposition}

\begin{proof}
The iterates $T^{2n}$ applied to $P_1P_2$ form a set of straight lines, equally distant from $O$, whose slopes are a dense subset of $\mathbb{R}$. Therefore, they envelop a disk.
\end{proof}

Now, we turn to the case where $\alpha$ is a rational multiple of $\pi$. Then the dynamical system becomes periodic. But not all rational multiples will give non-trivial examples. 

Consider for instance $\alpha=\pi/3$ (see left of Figure \ref{Fig:periodic-trajectories}). We see that $P_kP_{k+1}$ is parallel to $P_{k+3}P_{k+4}$ for all $k$ and both lie on the same side of $O$, independently of the choice of $\beta_1$. Since $S$ has to be tangent to all these lines, the only possibility is $P_kP_{k+1}=P_{k+3}P_{k+4}$, which implies $\beta_k=\beta_{k+1}=\alpha/2$. This argument applies to all points on $C$, so $\beta(P)=\alpha/2$ for all $P\in C$. By rotational symmetry, $S$ has to be a disk then.

The generalization of this argument leads to the following:

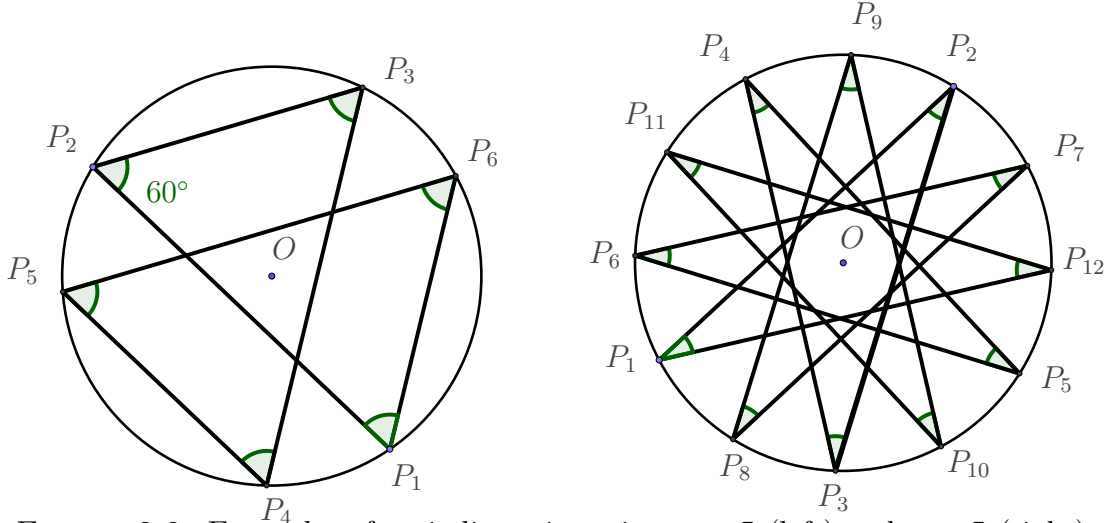
\begin{figure}
\definecolor{uuuuuu}{rgb}{0.26666666666666666,0.26666666666666666,0.26666666666666666}
\definecolor{qqwuqq}{rgb}{0,0.39215686274509803,0}
\definecolor{xdxdff}{rgb}{0.49019607843137253,0.49019607843137253,1}
\definecolor{ududff}{rgb}{0.30196078431372547,0.30196078431372547,1}

\begin{tikzpicture}[line cap=round,line join=round,>=triangle 45,x=1cm,y=1cm, scale=0.45]
%\clip(-26.718730712178125,-10.710026367334388) rectangle (8.541889786970144,5.999552012339383);
\draw [shift={(1.9490665961736235,3.357517122320277)},line width=1.5pt,color=qqwuqq,fill=qqwuqq,fill opacity=0.10000000149011612] (0,0) -- (-137.07156017018377:1.0230354110004332) arc (-137.07156017018377:-107.07156017018379:1.0230354110004332) -- cycle;
\draw [shift={(-1.5410663309166677,-8.007406636193567)},line width=1.5pt,color=qqwuqq,fill=qqwuqq,fill opacity=0.10000000149011612] (0,0) -- (72.92843982981624:1.0230354110004332) arc (72.92843982981624:102.92843982981623:1.0230354110004332) -- cycle;
\draw [shift={(-4.198969074522828,3.571193026071778)},line width=1.5pt,color=qqwuqq,fill=qqwuqq,fill opacity=0.10000000149011612] (0,0) -- (-77.07156017018377:1.0230354110004332) arc (-77.07156017018377:-47.07156017018378:1.0230354110004332) -- cycle;
\draw [shift={(3.8982771488783365,-5.133812630629879)},line width=1.5pt,color=qqwuqq,fill=qqwuqq,fill opacity=0.10000000149011612] (0,0) -- (132.92843982981623:1.0230354110004332) arc (132.92843982981623:162.9284398298162:1.0230354110004332) -- cycle;
\draw [shift={(-7.458035670696452,-1.6463240962484975)},line width=1.5pt,color=qqwuqq,fill=qqwuqq,fill opacity=0.10000000149011612] (0,0) -- (-17.071560170183783:1.0230354110004332) arc (-17.071560170183783:12.92843982981621:1.0230354110004332) -- cycle;
\draw [shift={(4.129343479795006,1.0135940055636865)},line width=1.5pt,color=qqwuqq,fill=qqwuqq,fill opacity=0.10000000149011612] (0,0) -- (-167.0715601701838:1.0230354110004332) arc (-167.0715601701838:-137.07156017018377:1.0230354110004332) -- cycle;
\draw [shift={(-4.569066596173626,-7.077517122320275)},line width=1.5pt,color=qqwuqq,fill=qqwuqq,fill opacity=0.10000000149011612] (0,0) -- (42.928439829816206:1.0230354110004332) arc (42.928439829816206:72.9284398298162:1.0230354110004332) -- cycle;
\draw [shift={(-1.0789336690833293,4.287406636193568)},line width=1.5pt,color=qqwuqq,fill=qqwuqq,fill opacity=0.10000000149011612] (0,0) -- (-107.0715601701838:1.0230354110004332) arc (-107.0715601701838:-77.07156017018382:1.0230354110004332) -- cycle;
\draw [shift={(1.5789690745228253,-7.291193026071779)},line width=1.5pt,color=qqwuqq,fill=qqwuqq,fill opacity=0.10000000149011612] (0,0) -- (102.9284398298162:1.0230354110004332) arc (102.9284398298162:132.9284398298162:1.0230354110004332) -- cycle;
\draw [shift={(-6.518277148878337,1.413812630629883)},line width=1.5pt,color=qqwuqq,fill=qqwuqq,fill opacity=0.10000000149011612] (0,0) -- (-47.07156017018381:1.0230354110004332) arc (-47.07156017018381:-17.071560170183815:1.0230354110004332) -- cycle;
\draw [shift={(4.838035670696454,-2.0736759037515045)},line width=1.5pt,color=qqwuqq,fill=qqwuqq,fill opacity=0.10000000149011612] (0,0) -- (162.9284398298162:1.0230354110004332) arc (162.9284398298162:192.92843982981617:1.0230354110004332) -- cycle;
\draw [shift={(-23.478089222303865,0.9749572498001466)},line width=1.5pt,color=qqwuqq,fill=qqwuqq,fill opacity=0.10000000149011612] (0,0) -- (-43.56475096667924:1.0230354110004332) arc (-43.56475096667924:16.435249033320737:1.0230354110004332) -- cycle;
\draw [shift={(-15.506770161650943,3.326373845392565)},line width=1.5pt,color=qqwuqq,fill=qqwuqq,fill opacity=0.10000000149011612] (0,0) -- (-163.56475096667924:1.0230354110004332) arc (-163.56475096667924:-103.56475096667927:1.0230354110004332) -- cycle;
\draw [shift={(-18.34654432770126,-8.443523348852539)},line width=1.5pt,color=qqwuqq,fill=qqwuqq,fill opacity=0.10000000149011612] (0,0) -- (76.43524903332073:1.0230354110004332) arc (76.43524903332073:136.43524903332073:1.0230354110004332) -- cycle;
\draw [shift={(-24.368590364691084,-2.7158668384521967)},line width=1.5pt,color=qqwuqq,fill=qqwuqq,fill opacity=0.10000000149011612] (0,0) -- (-43.56475096667927:1.0230354110004332) arc (-43.56475096667927:16.435249033320726:1.0230354110004332) -- cycle;
\draw [shift={(-12.755673311518471,0.7097651898600121)},line width=1.5pt,color=qqwuqq,fill=qqwuqq,fill opacity=0.10000000149011612] (0,0) -- (-163.56475096667927:1.0230354110004332) arc (-163.56475096667927:-103.56475096667928:1.0230354110004332) -- cycle;
\draw [line width=1pt] (-1.31,-1.86) circle (6.151747719144536cm);
\draw [line width=1.5pt] (-6.749343479795005,-4.733594005563689)-- (1.9490665961736235,3.357517122320277);
\draw [line width=2pt] (1.9490665961736235,3.357517122320277)-- (-1.5410663309166677,-8.007406636193567);
\draw [line width=1.5pt] (-1.5410663309166677,-8.007406636193567)-- (-4.198969074522828,3.571193026071778);
\draw [line width=1.5pt] (-4.198969074522828,3.571193026071778)-- (3.8982771488783365,-5.133812630629879);
\draw [line width=1.5pt] (3.8982771488783365,-5.133812630629879)-- (-7.458035670696452,-1.6463240962484975);
\draw [line width=1.5pt] (-7.458035670696452,-1.6463240962484975)-- (4.129343479795006,1.0135940055636865);
\draw [line width=1.5pt] (4.129343479795006,1.0135940055636865)-- (-4.569066596173626,-7.077517122320275);
\draw [line width=1.5pt] (-4.569066596173626,-7.077517122320275)-- (-1.0789336690833293,4.287406636193568);
\draw [line width=1.5pt] (-1.0789336690833293,4.287406636193568)-- (1.5789690745228253,-7.291193026071779);
\draw [line width=1.5pt] (1.5789690745228253,-7.291193026071779)-- (-6.518277148878337,1.413812630629883);
\draw [line width=1.5pt] (-6.518277148878337,1.413812630629883)-- (4.838035670696454,-2.0736759037515045);
\draw [line width=1.5pt] (4.838035670696454,-2.0736759037515045)-- (-6.749343479795005,-4.733594005563689);
\draw [line width=1pt] (-18.193435620507863,-2.2529336363974632) circle (6.1924828028968335cm);
\draw [line width=1.5pt] (-14.704946335181575,-7.369307916132743)-- (-23.478089222303865,0.9749572498001466);
\draw [line width=1.5pt] (-23.478089222303865,0.9749572498001466)-- (-15.506770161650943,3.326373845392565);
\draw [line width=1.5pt] (-15.506770161650943,3.326373845392565)-- (-18.34654432770126,-8.443523348852539);
\draw [line width=1.5pt] (-18.34654432770126,-8.443523348852539)-- (-24.368590364691084,-2.7158668384521967);
\draw [line width=1.5pt] (-24.368590364691084,-2.7158668384521967)-- (-12.755673311518471,0.7097651898600121);
\draw [line width=1.5pt] (-12.755673311518471,0.7097651898600121)-- (-14.704946335181575,-7.369307916132743);
\draw [shift={(-6.749343479795005,-4.733594005563689)},line width=1.5pt,color=qqwuqq,fill=qqwuqq,fill opacity=0.10000000149011612] (0,0) -- (12.928439829816215:1.0230354110004332) arc (12.928439829816215:42.92843982981624:1.0230354110004332) -- cycle;
\draw [shift={(-14.704946335181575,-7.369307916132743)},line width=1.5pt,color=qqwuqq,fill=qqwuqq,fill opacity=0.10000000149011612] (0,0) -- (76.43524903332079:1.0230354110004332) arc (76.43524903332079:136.43524903332073:1.0230354110004332) -- cycle;

\draw [fill=ududff] (-1.31,-1.86) circle (2.5pt);
\draw[color=uuuuuu] (-1.0405418960672517,-1.1105440941136402) node {$O$};
\draw [fill=xdxdff] (-6.749343479795005,-4.733594005563689) circle (2.5pt);
\draw[color=uuuuuu] (-7.860777969403474,-4.7) node {$P_1$};
\draw [fill=xdxdff] (1.9490665961736235,3.357517122320277) circle (2.5pt);
\draw[color=uuuuuu] (2.1990702387674537,4.5) node {$P_2$};
%\draw[color=qqwuqq] (3.153903289034525,3.15) node {$ 30\textrm{\degre}$};
\draw [fill=uuuuuu] (-1.5410663309166677,-8.007406636193567) circle (2pt);
\draw[color=uuuuuu] (-1.6,-8.8) node {$P_3$};
\draw [fill=uuuuuu] (-4.198969074522828,3.571193026071778) circle (2pt);
\draw[color=uuuuuu] (-5.1,4.5) node {$P_4$};
\draw [fill=uuuuuu] (3.8982771488783365,-5.133812630629879) circle (2pt);
\draw[color=uuuuuu] (5,-5.3) node {$P_5$};
\draw [fill=uuuuuu] (-7.458035670696452,-1.6463240962484975) circle (2pt);
\draw[color=uuuuuu] (-8.35,-1.6220617996138575) node {$P_6$};
\draw [fill=uuuuuu] (4.129343479795006,1.0135940055636865) circle (2pt);
\draw[color=uuuuuu] (5.370480012868796,1.5493479744874907) node {$P_7$};
\draw [fill=uuuuuu] (-4.569066596173626,-7.077517122320275) circle (2pt);
\draw[color=uuuuuu] (-4.5,-7.9) node {$P_8$};
\draw [fill=uuuuuu] (-1.0789336690833293,4.287406636193568) circle (2pt);
\draw[color=uuuuuu] (-0.6,5.4368825362891435) node {$P_9$};
\draw [fill=uuuuuu] (1.5789690745228253,-7.291193026071779) circle (2pt);
\draw[color=uuuuuu] (2.4,-7.8) node {$P_{10}$};
\draw [fill=uuuuuu] (-6.518277148878337,1.413812630629883) circle (2pt);
\draw[color=uuuuuu] (-7.144653181703171,2.5) node {$P_{11}$};
\draw [fill=uuuuuu] (4.838035670696454,-2.0736759037515045) circle (2pt);
\draw[color=uuuuuu] (5.8,-1.8) node {$P_{12}$};
\draw [fill=ududff] (-18.193435620507863,-2.2529336363974632) circle (2.5pt);
\draw[color=uuuuuu] (-17.81832263647436,-1.4174547174137706) node {$O$};
\draw [fill=xdxdff] (-14.704946335181575,-7.369307916132743) circle (2.5pt);

\draw[color=uuuuuu] (-14.2,-8.2) node {$P_{1}$};
\draw [fill=xdxdff] (-23.478089222303865,0.9749572498001466) circle (2.5pt);
\draw[color=uuuuuu] (-24.399850447243814,1.8) node {$P_{2}$};
\draw[color=qqwuqq] (-21.26254185350915,0.25) node {$ 60\textrm{\degre}$};
\draw [fill=uuuuuu] (-15.506770161650943,3.326373845392565) circle (2pt);
\draw[color=uuuuuu] (-14.408204599806247,3.8341270590551293) node {$P_{3}$};
\draw [fill=uuuuuu] (-18.34654432770126,-8.443523348852539) circle (2pt);
\draw[color=uuuuuu] (-18.057030899041127,-9.2) node {$P_{4}$};
\draw [fill=uuuuuu] (-24.368590364691084,-2.7158668384521967) circle (2pt);
\draw[color=uuuuuu] (-25.55929057971097,-2.1) node {$P_{5}$};
\draw [fill=uuuuuu] (-12.755673311518471,0.7097651898600121) circle (2pt);
\draw[color=uuuuuu] (-11.952919613405207,1.4) node {$P_{6}$};
\end{tikzpicture}

\vspace*{-0.8cm}
\caption{Examples of periodic trajectories, $\alpha=\tfrac{\pi}{3}$ (left) and $\alpha=\tfrac{\pi}{6}$ (right).}\label{Fig:periodic-trajectories}
\end{figure}

\begin{proposition}
If $\alpha=\tfrac{p}{q}\pi$ with $q-p$ even (and $p$ and $q$ coprime), then the shape $S$ is a disk.
\end{proposition}

\begin{proof}
First note that since $p$ and $q$ are coprime and $q-p$ is even, both are odd. Denote by $Q_k$ the orthogonal projection of $O$ onto $P_kP_{k+1}$. The angle between $OQ_k$ and $OQ_{k+1}$ is always $\pi-\alpha = \tfrac{q-p}{q}\pi$. Hence, the angle between $OQ_{k}$ and $OQ_{k+q}$ is $q(\pi-\alpha)=(q-p)\pi$. Since $q-p$ is even, $P_kP_{k+1}$ is parallel to $P_{k+q}P_{k+q+1}$ and both lie on the same side of $O$. These two parallels can only be tangent to the same shape if they coincide: $P_kP_{k+1}=P_{k+q}P_{k+q+1}$. It then follows $\beta_k=\beta_{k+q}=\beta_{k+1}$, since $q$ is odd and $\beta_k$ is 2-periodic. Hence $\beta(P)=\alpha/2$ for all points $P\in C$. By rotational symmetry, $S$ has to be a disk.
\end{proof}

\begin{proposition}
If $\alpha=\tfrac{p}{q}\pi$ with $q-p$ odd (and $p$ and $q$ coprime), then there are infinitely many possible shapes $S$ with $C$ as isoptics $S_\alpha$.
\end{proposition}

\begin{proof}
Start from the special case $p=q-1$. Then the trajectory of the dynamical system is a regular $2q$-gon when the initial condition is $\beta(P)=\alpha/2$.

\begin{figure}[h!]
\definecolor{uuuuuu}{rgb}{0.26666666666666666,0.26666666666666666,0.26666666666666666}
\definecolor{qqwuqq}{rgb}{0,0.39215686274509803,0}
\definecolor{xdxdff}{rgb}{0.49019607843137253,0.49019607843137253,1}
\definecolor{ffqqqq}{rgb}{1,0,0}

\begin{tikzpicture}[line cap=round,line join=round,>=triangle 45,x=1cm,y=1cm, scale=0.6]
%\clip(-5.686470995670994,-5.066832900432892) rectangle (9.430412121212116,5.599833766233769);
\draw[line width=1.5pt,color=qqwuqq,fill=qqwuqq,fill opacity=0.10000000149011612] (0.8478946097446765,4.213495356293595) -- (1.1132932273067249,3.9595394962390795) -- (1.3672490873612402,4.224938113801128) -- (1.1018504697991918,4.4788939738556435) -- cycle; 
\draw [shift={(-2.024847300228368,-3.013540432951146)},line width=1.5pt,color=qqwuqq,fill=qqwuqq,fill opacity=0.10000000149011612] (0,0) -- (42.70306033083914:0.8) arc (42.70306033083914:92.66855782966967:0.8) -- cycle;
\draw [rotate around={-2.915322116642519:(1.2572519480519466,0.015418181818185746)},line width=1.5pt,color=ffqqqq] (1.2572519480519466,0.015418181818185746) ellipse (3.424413949238913cm and 2.867081245762928cm);
\draw [line width=1pt] (1.2572519480519464,0.015418181818185744) circle (4.466180220898789cm);
\draw [line width=1.5pt] (1.1018504697991918,4.4788939738556435)-- (5.723240854016704,0.05675650964206991);
\draw [line width=1.5pt] (5.723240854016704,0.05675650964206991)-- (1.4126534263047021,-4.448057610219272);
\draw [line width=1.5pt] (1.4126534263047021,-4.448057610219272)-- (-3.2087369579128096,-0.025920146005698752);
\draw [line width=1.5pt] (-3.2087369579128096,-0.025920146005698752)-- (1.1018504697991918,4.4788939738556435);
\draw [line width=1.5pt] (-2.024847300228368,-3.013540432951146)-- (-2.292358628955975,2.7259588137108426);
\draw [line width=1.5pt] (-2.292358628955975,2.7259588137108426)-- (4.539351196332261,3.0443767965875175);
\draw [line width=1.5pt] (4.539351196332261,3.0443767965875175)-- (4.806862525059868,-2.695122450074471);
\draw [line width=1.5pt] (4.806862525059868,-2.695122450074471)-- (-2.024847300228368,-3.013540432951146);
\draw [line width=1.5pt] (-2.024847300228368,-3.013540432951146)-- (1.2572519480519468,0.015418181818185571);
\draw (-1.8076831168831167,-0.3222441558441497) node[anchor=north west] {$\beta(P)$};
\draw [fill=xdxdff] (1.1018504697991918,4.4788939738556435) circle (2.5pt);
\draw[color=uuuuuu] (1.4390701298701287,5.253513419913422) node {$P_3$};
\draw [fill=uuuuuu] (5.723240854016704,0.05675650964206991) circle (2pt);
\draw[color=uuuuuu] (6.270238961038958,0.30113246753247325) node {$P_4$};
\draw [fill=uuuuuu] (-3.2087369579128096,-0.025920146005698752) circle (2pt);
\draw[color=uuuuuu] (-3.859631168831168,0.11065627705628289) node {$P_2$};
\draw [fill=uuuuuu] (1.4126534263047021,-4.448057610219272) circle (2pt);
\draw[color=uuuuuu] (1.8,-5) node {$P_1$};
\draw [fill=xdxdff] (-2.024847300228368,-3.013540432951146) circle (2.5pt);
\draw[color=uuuuuu] (-2.6561679653679646,-3.058174891774884) node {$P$};
\draw [fill=uuuuuu] (1.2572519480519468,0.015418181818185571) circle (2pt);
\draw[color=uuuuuu] (1.4304121212121204,0.43966060606061164) node {$O$};
\end{tikzpicture}

\caption{Construction of non-trivial constant angle shapes}\label{Fig:proof}
\end{figure}
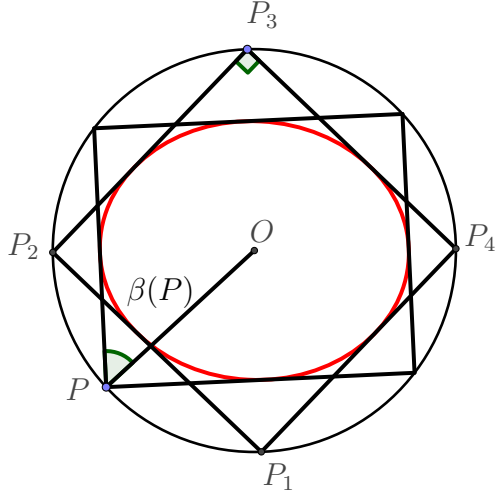

Consider a regular $2q$-gon inscribed in the circle $C$, with endpoints $P_1, P_2, ..., P_{2q}$. 
Let $\theta: \widehat{P_1P_2}\to \mathbb{R}$ be a smooth function on the arc $\widehat{P_1P_2}$ such that $\theta(P_1)=\theta(P_2)=0$ and $\theta$ is $\mathcal{C}^2$-close to the zero-function (i.e. the supremum norm of the values of $\theta, \theta'$ and $\theta''$ are small). For $P$ on the arc $\widehat{P_1P_2}$, take $\beta(P)=\theta(P)+\alpha/2$.

The function $\beta(P)$ can then be extended to all of $C$ using the dynamical system. 
On the second arc $P\in\widehat{P_2P_3}$ it is given by $\beta(P)=-\theta(P')+\alpha/2$, where $P'\in\widehat{P_1P_2}$ is the image of $P$ under the rotation around the origin of angle $-\pi/q$. 
By Observation \ref{obs-2-periodic}, the function $\beta$ stays invariant under rotation around the origin by an angle of $2\pi/n$. Hence we know it everywhere on $C$.

In particular, we see that $\beta$ stays $\mathcal{C}^2$-close to the zero function in the interior of all arcs $\widehat{P_kP_{k+1}}$. The only concern is that $\beta$ might not be $\mathcal{C}^2$ at the junction. To remedy this, we impose $\theta''(P_1)=\theta''(P_2)=0$ on the initial arc $\widehat{P_1P_2}$, making the function $\beta$ defined on all of $C$ of class $\mathcal{C}^2$.

The union of all trajectorires of the dynamical system with initial condition $(P, \beta(P))$, where $P$ varies on $\widehat{P_1P_2}$, envelop a convex shape $S$. This is true for $\theta$ identically 0 (where $S$ is a disk), and it stays true for a smooth variation which is $\mathcal{C}^2$-close to the zero function. We admit this result which can be proven using support functions (see \cite{green}).

In the general case, for $p\neq q-1$, the trajectory of the dynamical system with initial condition $\beta(P)=\alpha/2$ is a star-shaped regular $2q$-gon. The same argument as above works, this time applied to the arc $\widehat{P_1P_k}$, where $P_k$ is one of the closest points to $P_1$.
\end{proof}

\begin{Remark}
There always exists a trajectory of the dynamical system which is a (potentially star-shaped) regular $2q$-gon. This follows from the intermediate value theorem, applied to the function $\beta(P)-\alpha/2$.
\end{Remark}

To finish, we give a class of examples, where each curve is analytic.

\begin{example}
Consider the unit circle $C$ and define $\theta(t)=k\sin(qt)$ for $t\in [0,\pi/q]$ and $k\in\mathbb{R}$ a sufficiently small constant. Then $\theta$ is $\mathcal{C}^2$-close to the zero function. Its extension to $C$ is still given by the same formula (for $t\in[0,2\pi]$), since it is $2\pi/q$-periodic and satisfies $\theta(t+\pi/q)=-\theta(t)$.

Figure \ref{Fig-examples} shows two examples, produced using SageMath. This class of examples appeared in Green's work \cite{green} using the support function. This allows in particular to write a parametric equations for the curve, showing the analyticity.

\begin{figure}[h!]
\centering
\includegraphics[height=5cm]{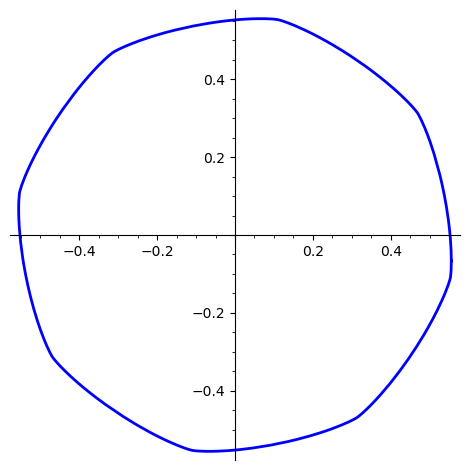} \hspace{2cm}
\includegraphics[height=5cm]{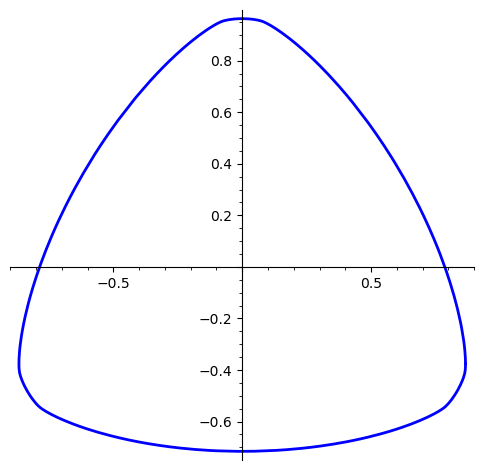}

\caption{Examples of analytic ``constant angle'' shapes. The parameters are $p=3,q=8,k=0.01$ (left) and $p=2,q=3,k=0.25$ (right).}\label{Fig-examples}
\end{figure}
\end{example}

\section{A strange billiards}\label{Sec:strange-billiards}

The dynamical viewpoint in the proof above can be generalized, leading to a dynamical system, which we call the \emph{outer constant angle billiards}.

Fix a convex compact set $C\subset \R^2$. A configuration of the dynamical system is a point $P$ outside of $C$ and a unit vector $v$ generating a tangent ray from $P$ to $C$. Let $\alpha$ be the angle of sight of $P$ to $C$. Then there is a unique point $P'$ on the tangent ray generated by $v$ such that the angle of sight at $P'$ to $C$ is $\alpha$. The dynamical map is $(P,v)\mapsto (P',v')$, where $v'$ is a unit vector generating a tangent ray disjoint from the tangent ray generated by $v$.

Equivalently, we can consider the isoptics $C_\alpha$ of $C$ with angle $\alpha$. Then $P'$ is the second intersection point of the tangent ray generated by $v$ with the isoptics $C_\alpha$.
Note that the interior of $C_\alpha$ is not necessarily convex, but it is star-shaped with respect to all points of $C$.

The outer constant angle billiards on $C$ is then equivalent to an \emph{inner constant angle billiards} on $C_\alpha$, by which we mean a piecewise straight trajectory inside of $C_\alpha$ which changes direction on $C_\alpha$ by an angle $\alpha$.

\begin{proposition}
The outer constant angle billiards is periodic if and only if $\alpha$ is a rational multiple of $\pi$.
\end{proposition}
\begin{proof}
If $\alpha$ is an irrational multiple of $\pi$, then the slopes of the tangent rays of the dynamical systems are all different, hence it cannot be periodic. If $\alpha$ is a rational multiple of $\pi$, then we get infinitely many tangent rays with the same slope. Since there are exactly two tangents with given slope to $C$, the system has to be periodic.
\end{proof}

Exploring this dynamical system would be an interesting task, notably studying  periodic orbits, unique ergodicity or phenomena like the Poncelet porism (see for instance Lecture 29 in \cite{fuchs2007mathematical}).

\section{Opening}\label{Sec:opening}

Many interesting questions can be asked around isoptics. For example:
\begin{question}
Are there convex shapes which have exactly two circular isoptics? These two circles have to be concentric?
\end{question}

\begin{question}\label{Q2}
For which closed convex set $S$ there is an isoptics which is a straight line?
\end{question}

It is not difficult to show that such a shape is unbounded (hence $\alpha\geq \pi/2$), and is the graph of a function over the straight line isoptics. The simplest example is a parabola and $\alpha=\pi/2$ (the orthoptic circle degenerates into a line). There are examples for all $\alpha \in [\pi/2,\pi)$, but a full classification of all shapes seems to be unknown.

\medskip
Our original Question \ref{Q1} and related questions like \ref{Q2} constituted the first problem of the first ETEAM, a new research-oriented team competition for high school students which took place in July 2024 at the University of Luxembourg. For more information about the competition, see \href{https://eteam.tfjm.org/}{https://eteam.tfjm.org/}. If you have ideas for problems please contact the scientific organization comitee.
%The next edition will be held in Lyon in July 2025. 

To finish, I give the full problem \emph{Exploring Flatland} as it was stated in the ETEAM 2024, very close to my initial suggestions. Good luck!

\medskip
\itshape
In Flatland the reality takes place inside the plane. The objects in Flatland are convex shapes.
The Flatlanders have each one punctual eye, which is capable of measuring angles.

At a point $P$ outside of an object $\mathcal{S}$, there are exactly two \emph{tangent rays} from $P$ to $\mathcal{S}$ (the rays through $P$ intersecting the boundary of $\mathcal{S}$ without intersecting its interior). The angle between the two tangent rays is called the \emph{angle of sight} at $P$.
The eye of a Flatlander allows one to determine the angle of sight.

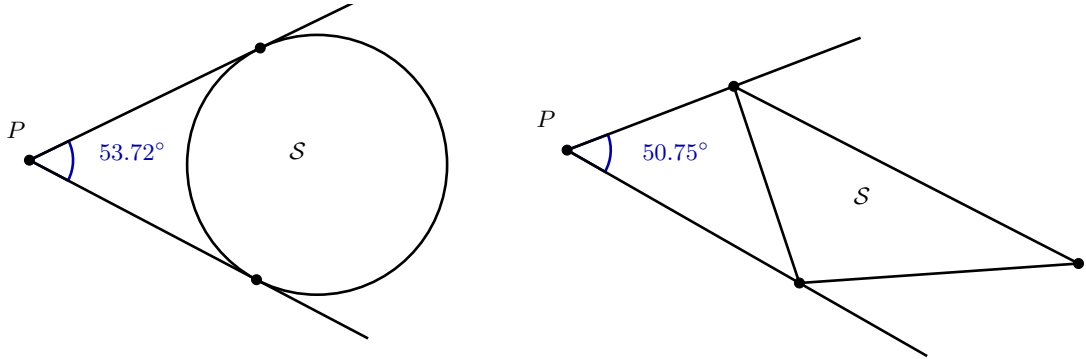
\begin{figure}[ht]
    \centering
 	\definecolor{qqwuqq}{rgb}{0,0.39215686274509803,0}
\begin{tikzpicture}[scale=0.75][line cap=round,line join=round,>=triangle 45,x=1cm,y=1cm]
\clip(-5.907689870816927,-2.5) rectangle (16.7346651039,4.);
\draw [shift={(-4.759630183997679,1.2419409591069186)},line width=1pt,color=blue!70!black] (0,0) -- (-27.77910700789548:0.7736567759020533) arc (-27.77910700789548:25.93800599598098:0.7736567759020533) -- cycle;
\draw [shift={(4.770909672580966,1.4206370338651853)},line width=1pt,color=blue!70!black] (0,0) -- (-29.7919608093646:0.7736567759020533) arc (-29.7919608093646:20.955776730632138:0.7736567759020533) -- cycle;
\draw [line width=1pt] (0.34,1.16) circle (2.30425693cm);
\draw [line width=1pt] (-4.7596301,1.241940959)-- (1.1733325,4.12769785);
\draw [line width=1pt] (-4.7596301,1.241940)-- (1.24308628,-1.920135);
\draw [line width=1pt] (7.725902,2.552336)-- (8.889038,-0.9370702);
\draw [line width=1pt] (8.889038,-0.937070)-- (13.8364158,-0.5912731);
\draw [line width=1pt] (13.83641,-0.591273)-- (7.72590278,2.552336);
\draw [line width=1pt] (4.770909,1.420637)-- (9.9737430,3.4132115);
\draw [line width=1pt] (4.77090,1.4206370)-- (11.152689,-2.233053);
\begin{scriptsize}
\draw[color=black] (0.,1.4) node {$\mathcal{S}$};
\draw[color=black] (10.,0.6) node {$\mathcal{S}$};
\draw [fill=black] (-4.759630183997679,1.2419409591069186) circle (2.5pt);
\draw[color=black] (-5.,1.7859264908573265) node {$P$};
\draw [fill=black] (-0.6678782844606234,3.23214414646103) circle (2.5pt); %D
\draw [fill=black] (-0.7339313068542868,-0.8786935885900662) circle (2.5pt); %E
\draw [fill=black] (8.88903837620901,-0.9370702392042988) circle (2.5pt); %I
\draw [fill=black] (13.83641580752044,-0.591273172487441) circle (2.5pt); %J
\draw [fill=black] (7.7259027881613935,2.552336524938538) circle (2.5pt); %H
\draw [fill=black] (4.770909672580966,1.4206370338651853) circle (2.5pt);
\draw[color=black] (4.4,1.9664464052344717) node {$P$};
\draw[color=blue!70!black] (-2.93,1.4) node {$53.72^{\circ}$};
\draw[color=blue!70!black] (6.7,1.4) node {$50.75^{\circ}$};
\end{scriptsize}
\end{tikzpicture}
    \caption{Example of two objects with angle of sight.}
\end{figure}

\q In the case where $\mathcal{S}$ is a disk with center at the origin and radius 1, what is the angle of sight at $P=(x,y)$ (outside of $\mathcal{S}$)?

\medskip
Edwin is a Flatlander who dislikes changing sight. Whenever he looks at an object, he goes around it by keeping the same angle of sight.

\q Determine the path Edwin must take to go around a square. Describe the path for any convex polygon.

\medskip
Emily is a Flatlander and an explorer. She wants to know the exact shape of the objects she encounters. 
To this end, Emily has a device which allows to determine tangent rays.

\q With a finite number of measurements, is it always possible for Emily to determine the exact shape
\begin{enumerate}
    \item of an arbitrary object?

    \item of an $n$-gon for some fixed $n\geq 3$ (which is known to Emily)? If yes, give a bound for how many measurements she has to do. One can start with the case of a triangle.
\end{enumerate}

\medskip
One day, Edwin and Emily walk together on a circle $C$, focusing on a bounded shape $\mathcal{S}$ under a constant angle. Emily forgot to take her device to determine tangent rays. She wonders whether $\mathcal{S}$ has to be a disk.

\q Under which conditions on the angle of sight $\alpha$, Emily can deduce that the shape is a disk? One can start by treating the following cases:

\begin{enumerate}
    \item $\alpha$ is an irrational multiple of $\pi$.

    \item $\alpha=\pi/3$.

    \item $\alpha$ is a rational multiple of $\pi$.
\end{enumerate}

\medskip
Naturally Edwin likes shapes where he can walk on a straight line $\ell$, while keeping the same angle of sight.

\q For given straight line $\ell$ and angle of sight $\alpha$, describe the space of all such shapes $\mathcal{S}$. In particular, consider the following cases:
\begin{enumerate} 
    \item $\mathcal{S}$ is bounded (i.e. of finite size).

    \item $\mathcal{S}$ is unbounded. One can start by considering $\alpha<\pi/2, \alpha=\pi/2$ and then $\alpha>\pi/2$. Are there examples where $\mathcal{S}$ has smooth boundary?
\end{enumerate}

\q Suggest and study further directions of research, for example in 3 dimensions.

\bibliographystyle{plain}
\bibliography{ref}

\begin{thebibliography}{1}

\bibitem{blaschke1915konvexe}
Wilhelm Blaschke.
\newblock {Konvexe Bereiche gegebener konstanter Breite und kleinsten Inhalts}.
\newblock {\em Mathematische Annalen}, 76(4):504--513, 1915.

\bibitem{fuchs2007mathematical}
Dmitry Fuchs and Serge Tabachnikov.
\newblock {\em Mathematical Omnibus: Thirty Lectures on Classic Mathematics}.
\newblock AMS, 2007.

\bibitem{green}
John~W. Green.
\newblock {Sets subtending a constant angle on a circle}.
\newblock {\em Duke Math. J.}, 17(3):263--267, 1950.

\bibitem{sidler}
Jean-Claude Sidler.
\newblock {\em G{\'e}om{\'e}trie projective: Cours, exercices et problemes
  corrig{\'e}s}.
\newblock Dunod, 2009.

\end{thebibliography}
%\nocite{*}

\end{document}